\documentclass{amsart}
\usepackage {amsmath, amscd}
\usepackage {amssymb}
\usepackage {epsfig}
\usepackage {bm}
\usepackage {indentfirst} 
\usepackage{color}

\usepackage{hyperref}

\usepackage[active]{srcltx}

\numberwithin{equation}{section}

\def\<{\langle}
\def\>{\rangle}

\def\HH{{\mathcal H}}

\def\LL{{\mathcal L}}

\def\TT{{\mathcal T}}

\def\bbC{\mathbb{C}}

\def\bbN{\mathbb{N}}

\def\bbD{\mathbb{D}}
\def\bbT{\mathbb{T}}

\def\1{\mathbf{1}}

\newtheorem{lemma}{Lemma}[section]
\newtheorem{theorem}[lemma]{Theorem}

\newtheorem{conjecture}[lemma]{Conjecture}

\theoremstyle{definition}

\title{Truncated Toeplitz operators and complex symmetries}
\author{Hari Bercovici and Dan Timotin}
\address{Mathematics Department, Indiana University, Bloomington, IN 47405, USA}
\email{bercovic@indiana.edu}
\address{Simion Stoilow Institute of Mathematics, Romanian Academy, Calea Grivi\c tei 21, Bucharest, Romania}
\email{dan.timotin@imar.ro}
\thanks{HB was supported in part by a grant of the National Science Foundation.}

\begin{document}
	
\begin{abstract}
	We show that truncated Toeplitz operators are characterized by a collection of complex symmetries.  This was conjectured by Kli\'s-Garlicka, \L anucha, and Ptak, and proved by them in some special cases. 
\end{abstract}

\maketitle
	
\section{Introduction}
	
The systematic study of truncated Toeplitz operators was initiated by Sarason~\cite{Sa}. Given an inner function $u$ on the unit disc, this class, denoted by  $\mathcal{T}_u$ consists of those bounded operators on $K_u=H^2\ominus uH^2$ that are compressions of multiplication operators operators. A recent survey of results in this area is contained in~\cite{GR}. 

Sarason observed that, while every operator in $\mathcal{T}_u$ is complex symmetric (relative to the natural conjugation on $K_u$; see \cite{GPu}), not every complex symmetric operator on $K_u$ belongs to $\mathcal{T}_u$. Operators in $\mathcal{T}_u$ satisfy additional complex symmetry conditions and the authors of \cite{KLP} conjectured that every operator on $K_u$ that satisfies these additional symmetries necessarily belongs to $\mathcal{T}_u$.  This conjecture is proved in \cite{KLP} in many cases in which $u$ is a Blaschke product.  The purpose of this note is to provide a proof of this conjecture for arbitrary inner functions $u$. In the case in which $u$ has at least one zero, it turns out that the operators in $\mathcal{T}_u$ are characterized by the fact that they satisfy just two complex symmetries.  In case $u$ is singular, one needs to require a countable collection of complex symmetries.

\section{Notation and preliminaries}

We denote by $\bbC$   the complex plane, by $\bbD=\{z\in\bbC:|z|<1\}$ the unit disc, and by $\bbT=\{z\in\bbC:|z|=1\}$ the unit circle.
As usual, we  view the Hardy space $H^2$ on $\bbD$ as a subspace of $L^2=L^2(\bbT)$ (relative to the normalized arclength measure on $\mathbb T$)  by identifying functions analytic in $\bbD$ with their radial limits (which exist almost everywhere). Similarly,   the algebra $H^\infty$ of bounded analytic functions in $\bbD$ can be viewed as a closed subalgebra of $L^\infty=L^\infty(\bbT)$. We denote by $S$ the  shift operator in $H^2$, defined by $(Sf)(z)=zf(z)$, $f\in H^2$, $z\in\mathbb D$.

A function $u\in H^\infty$ is said to be \emph{inner} if $|u|=1$ almost everywhere on $\bbT$. For instance, the function $\chi\in H^\infty$ defined by $\chi(z)=z$, $z\in\mathbb D$, is inner.
  If $u$ is an inner function, the model space ${K_u}$ (often denoted $\mathcal {H}(u)$ in the literature) is defined by ${K_u}=H^2\ominus uH^2$ and $P_{{K_u}}:L^2\to K_u$ denotes the orthogonal projection onto ${K_u}$.

 Given an arbitrary bounded operator
  $A$ on a Hilbert space $\HH$, we  denote by $Q_A$ the quadratic form on $\HH$ defined by $Q_A(f)=\<Af, f\>$, $f\in\mathcal H$. A {\em conjugation} on a Hilbert space $\HH$ is an isometric, conjugate linear involution, that is, $C\circ C=I_{\mathcal H}$ and $\langle Ch,Ck\rangle=\langle k,h\rangle$ for $h,k\in\mathcal H$. An operator $A$ is then said to be {\em $C$-symmetric}~\cite{GPu} or simply \emph{complex symmetric} when $C$ is understood, if $A^*=CAC$. This condition is easily seen to be equivalent to $Q_A(f)=Q_A(Cf)$, $f\in\HH$.

Given an arbitrary inner function $u$, there is a conjugation $C_u$ on $L^2$ defined by $ C_uf= u\overline{\chi f}$.  This conjugation maps $K_u$ bijectively onto itself and therefore it also defines a conjugation on this space. We record for further use the following result whose proof is a simple calculation.

\begin{lemma}\label{le:basic}
	Suppose that $u$ and $v$ are inner functions in $H^\infty$ and $v$ divides $u$. Then for every $f\in L^2$ we have
	\[
	C_u(C_{u/v} (f))= vf.
	\]
\end{lemma}

The space ${K_u}$ is a reproducing kernel space of analytic functions on $\bbD$.
The following well-known lemma is the $H^2$ version of a results that holds in arbitrary reproducing kernel Hilbert spaces.

\begin{lemma}\label{le:weak convergence general} Suppose that $\{f_n\}_{n\in\mathbb N}\subset H^2$.  Then:
 \begin{itemize}
   \item[(i)] The sequence $\{f_n\}_{n\in\mathbb N}$ converges weakly to a function $f\in H^2$ if and only if $\sup_{n\in\mathbb N}\|f_n\|<+\infty$ and $\lim_{n\to\infty}f_n(z)= f(z)$ for every $z\in\bbD$.
  \item[(ii)] The sequence $\{f_n\}_{n\in\mathbb N}$ converges in norm to a function $f\in H^2$ if and only if $\lim_{n\to\infty}\|f_n\|=\|f\|$ and $\lim_{n\to\infty}f_n(z)= f(z)$ for every $z\in\bbD$.
 \end{itemize}
 \end{lemma}

We recall  \cite{Sa} that a bounded linear operator $A$ on $K_u$ is called a \emph{truncated Toeplitz operator} if there exists a function $\varphi\in L^2$ (called a \emph{symbol} of $A$) such that $$Af=P_{K_u}(\varphi f)$$ for every bounded function $f\in K_u$.  The truncated Toeplitz operators on ${K_u}$ form a weakly closed subspace $\mathcal {T}_u$ of $\LL({K_u})$. 
There is a simple characterization of the operators in $\mathcal{T}_u$ that does not require a symbol. The space
\[
K_u^0=\{ g\in K_u: Sg\in K_u\}.
\]
is closed in $K_u$ and it has codimension 1. Its orthogonal complement is generated the vector by $S^*u=\overline{\chi}(u-u(0))$. The following result is \cite[Theorem 8.1]{Sa}.

 \begin{lemma}\label{le:characterization of TTOs}
  A bounded linear operator $A$ on ${K_u}$ belongs to $\TT_u$ if and only if
  \begin{equation}\label{eq:characterization}
    Q_A(f)=Q_A(Sf)
  \end{equation}
  for every $f\in K_u^0$.
   \end{lemma}

  Fix $a\in\mathbb D$ and denote by
  $
  b_a(z)=({z-a})/({1-\bar a z})
  $, $z\in\mathbb D$, the corresponding Blaschke factor.
The following result is used in~\cite[Section 4]{CGRW} as well as \cite{KLP}.

\begin{lemma}\label{le:omega_a}
There is a unitary operator $\omega_a:K_u\to K_{u\circ b_a}$  defined by
\begin{equation}\label{eq:omega_a}
  \omega_a(f)=\frac{\sqrt{1-|a|^2}}{1-\bar a \chi} f\circ b_a,\quad f\in K_u.
\end{equation}
We have 
$
\omega_a C_u=C_{u\circ b_a} \omega_a,
$
and $\omega_a\TT_u\omega_a^*=\TT_{u\circ b_a}$.
\end{lemma}

\section{Truncated Toeplitz operators and conjugations}

Suppose that $u$ is an inner function and $A\in \mathcal{T}_u$. Then $A$ is $C_u$ symmetric, that is, $Q_A(f)=Q_A(C_uf)$ for every $f\in K_u$. Let $v$ be an inner divisor of the function $u$.  Then $K_v\subset K_u$ and it was observed in \cite{KLP} that $P_vA|K_v$ is also $C_v$-symmetric.  The authors of
\cite{KLP} formulate the following conjecture.

\begin{conjecture}\label{conjecture}
A bounded linear operator $A$ on $K_u$ belongs to $\TT_u$ if and only if, for every inner divisor $v$ of $u$, the compression $P_vA|K_v$ is $C_v$-symmetric.
\end{conjecture}

This conjecture is proved in~\cite{KLP} for certain Blaschke products $u$, namely, Blaschke products with a single zero, finite Blaschke products with simple zeros, and interpolating Blaschke products. The arguments rely on a characterization \cite{CRW} of the class $\mathcal{T}_u$ in terms of its matrix entries in a particular orthonormal basis for $K_u$.
In this section, we prove the conjecture for those inner functions $u$ that have at least one zero.  The case of singular inner functions is treated in the following section.

\begin{theorem}\label{th:with zero}
Suppose that $u\in H^\infty$ is an inner function and $u(a)=0$ for some $a\in\bbD$. Then an operator $A\in\LL(K_u)$ belongs to $\TT_u$ if and only if it is $C_u$-symmetric and $Q_A(C_{u/b_a}f)=Q_A(f)$ for every $f\in K_{u/b_a}$.
\end{theorem}

\begin{proof}
Suppose first that   $a=0$ and thus $b_a=\chi$.  If $f\in K_u$, then $Sf\in K_u$ if and only if $f\in K_{u/\chi}$. For such a function $f$ we have  $C_uC_{u/\chi}f=\chi f=Sf$ by Lemma~\ref{le:basic}. The two symmetry hypotheses in the statement imply that
\[
Q_A(Sf)=Q_A(C_uC_{u/z}f)=Q_A(C_{u/\chi}f))=Q_A(f).
\]
It follows then from Lemma~\ref{le:characterization of TTOs}  that $A\in \TT_u$.

For the  general case $a\ne0$ we use Lemma~\ref{le:omega_a}. The inner function $v=u\circ b_{-a}$ satisfies $v(0)=0$, and the unitary map $\omega_a$ defined in~\eqref{eq:omega_a} yields by restriction unitary maps from $K_v$ onto $K_u$ and from $K_{v/\chi}$ to $K_{u/b_a}$ that intertwine the standard conjugations on these spaces. Therefore $\omega_a A \omega_a^*$ is $C_v$-symmetric, and its compression to $K_{v/\chi}$ is $C_{v/\chi}$-symmetric. By the first part of the proof, $\omega_a A \omega_a^*\in \TT_v$. It follows from Lemma~\ref{le:omega_a} that $A\in\TT_u$.
\end{proof}

We have thus proved a stronger version of the conjecture in case $u$ has a zero in $\bbD$: an operator $A$ on $K_u$ is a truncated Toeplitz operator if and only if the complex symmetry condition is satisfied by $A$ as well as by a single one of its compressions to model spaces.

\section{Singular inner functions}

Given a positive, singular Borel measure  $\nu$  on $\bbT$, we denote by $e_\nu$ the corresponding singular inner function, that is,
\begin{equation}\label{eq:definition of e_nu}
e_\nu(z) = \exp\left( -\int_\bbT \frac{\zeta+z}{\zeta-z} \, d\nu(\zeta)
\right),\quad z\in\mathbb D.
\end{equation}

\begin{lemma}\label{le:convergence}
Let $\nu$ be a nonzero, positive, singular Borel measure on $\bbT$. Then there exist $\zeta\in\mathbb T$ and a sequence of nonzero, positive Borel measures $\mu_n\le \nu$, $n\in\mathbb N$, such that:
\begin{itemize}
\item[(i)]
$\lim_{n\to\infty}e_{\mu_n}(z)=1$ for every  $z\in\bbD$, and
\item[(ii)] for every $g\in H^2$, the functions
\[
\frac{e_{\mu_n}-1}{\mu_n(\bbT)} (\chi-\eta)g,\quad n\in\mathbb N,
\]
converge weakly in $H^2$ to $(\chi+\eta)g$ as $n\to\infty$.
\end{itemize}
\end{lemma}

\begin{proof}
Choose $\eta\in\bbT$ and a sequence $\{I_n \}_{n\in\mathbb N}$ of arcs in $\mathbb T$, symmetric about $\eta$, with length $|I_n|=1/n$, such that  $\lim_{n\to\infty}(\nu(I_n)/|I_n|)=+\infty $. This is possible since $\nu$ is singular. Define the measures $\mu_n$ by
\[
d\mu_n=\sqrt{\frac{|I_n|}{\nu(I_n)}}\, \chi_{I_n}\,d\nu,\quad n\in\mathbb N.
\]
By the maximum modulus principle, condition (i) only needs to be verified at $z=0$, and this is immediate because $e_{\mu_n}(0)=e^{-\mu_n (\mathbb T)}$. In fact, we have
\[
\lim_{n\to\infty}\frac{e_{\mu_n}(z)-1}{\mu_n(\bbT)}=  \frac{z+\eta}{z-\eta},\quad z\in\mathbb D.
\]
Lemma~\ref{le:weak convergence general} shows that (ii) is true as well once we verify that
\begin{equation}\label{eq:uniform boundedness}
\sup_{z\in\bbD, n\in\bbN} |z-\eta| \frac{|e_{\mu_n}(z)-1|}{\mu_n(\bbT)}<\infty.
\end{equation}
Observe first that, if $z\in\mathbb D$  and $|z-\eta|<10/n=10|I_n|$,
\[
|z-\eta| \frac{|e_{\mu_n}(z)-1|}{\mu_n(\bbT)}\le \frac{20|I_n|}{\sqrt{|I_n|\nu(I_n)}}
=20 \sqrt{\frac{|I_n|}{\nu(I_n)}}, 
\]
and the last quantity tends to $0$ by the choice of $I_n$.  If $|z-\eta|\ge10/n$, we use the inequalities
$$
|e^\lambda-1|\le |\lambda|e^{|\lambda|}, \quad\lambda\in\mathbb C,
$$
and
$$
\left|\frac{\zeta-z}{\zeta+z}\right|<3,\quad \zeta\in I_n, z\in\mathbb{D}, |z-\zeta|>\frac{10}{n},
$$
to deduce that
$$
|e_{\mu_n}(z)-1|\le 3\mu_n(\mathbb T)e^{3\mu_n(\mathbb T)}.
$$
For such values of $z$ we see that
\[
|z-\eta| \frac{|e_{\mu_n}(z)-1|}{\mu_n(\bbT)}\le 6e^{3\mu_n(\mathbb T)}<6e^{3\nu(\mathbb T)}.
\]
This concludes the proof of the lemma.
\end{proof}

We need one more technical result before establishing Conjecture \ref{conjecture} for $u=e_\nu$. Recall that $K_u^0$ consists of those vectors $f\in K_u$ with the property that $Sf$ also belongs to $K_u$. Clearly, $K_v^0\subset K_u^0$ if $v$ is an inner divisor of $u$.

\begin{lemma}\label{le:approximation}
Suppose that $u\in H^\infty$ is an inner function and $\{u_n\}_{n\in\mathbb N} $ is a sequence of inner divisors of $u$ such that $u_{n+1}$ divides $u_n$, $n\in\mathbb N$, and $\lim_{n\to\infty}|u_n(0)|=1$. Then $\bigcup_{n\in\mathbb N}K_{u/u_n}^0$ is dense in $K_u^0$.
\end{lemma}

\begin{proof}We can, and do, assume without loss of generality that $u_n(0)\ge0$, $n\in\mathbb N$. The least common inner multiple of the functions $\{u/u_n\}_{n\in\mathbb N}$ is equal to $u$, and thus $\bigcap_{n\in\mathbb N}(u/u_n)H^2=uH^2$. It follows that  $\bigcup_{n\in\mathbb N}K_{u/u_n}$ is dense in $K_u$ and therefore the sequence $\{P_{u/u_n}\}_{n\in\mathbb N}$ converges to $P_u$ in the strong operator topology.
	
It was noted earlier that the space $K_{u/u_n} \ominus K_{u/u_n}^0$ is generated by the vector $S^*(u/u_n)$. It follows from Lemma \ref{le:weak convergence general} that the sequence $\{u/u_n\}_{n\in\mathbb N}$ converges in the $H^2$ norm to $u$, and thus $\lim_{n\to\infty}S^*(u/u_n)=S^*u$.  If we denote by $P_n$ and $P$ the orthogonal projections onto the spaces $K_{u/u_n} \ominus K_{u/u_n}^0$ and $K_u\ominus K_u^0$, respectively, it follows that the sequence $\{P_n\}_{n\in\mathbb N}$ converges to $P$ in the strong operator topology. Given an arbitrary vector $f\in K_u ^0$, we have $f_n=P_{u/u_n}f-P_nP_{u/u_n}f\in K_{u/u_n}^0$ and $\lim_{n\to\infty}f_n=f-Pf=f$. The lemma follows. 
\end{proof}

We may now give the solution of the conjecture for singular inner functions.

\begin{theorem}\label{th:main singular}
Suppose that $\nu$ is a positive, singular Borel measure on $\bbT$. Let $A$ be an operator on $K_{e_\nu}$ that is $C_{e_\nu}$-symmetric, and such that for every positive Borel measure $\mu\le \nu$, the compression of $A$ to  $K_{e_\mu}$,  is  $C_{e_\mu}$-symmetric. Then $A\in\TT_{e_\nu}$.
\end{theorem}

\begin{proof}
Let  $\eta\in\bbT$ and   $\{\mu_n\}_{n\in\mathbb N}$ be as in Lemma~\ref{le:convergence}. By Lemmas \ref{eq:characterization} and \ref{le:approximation}, it suffices to show that $Q_A(Sf)=Q_A(f)$ for every $f\in\bigcup_{n\in\mathbb N}K_{e_\mu/e_{\mu_n}}$.  Fix $n\in \bbN$,  $f\in K^0_{e_\nu/e_{\mu_n}}$, and observe that then $(\chi-\eta)f\in K_{e_\nu/e_{\mu_n}}$. If  $m\ge n$, we also have $e_{\mu_m}|e_{\mu_n}$ and $(\chi-\eta)f\in K_{e_\nu/e_{\mu_m}}$.
 Lemma~\ref{le:basic} yields
 \[
 C_{e_\nu} (C_{e_\nu/{e_{\mu_m}}} ((\chi-\eta) f)= e_{\mu_m} (\chi-\eta) f.
 \]
The complex symmetry of $A$ and of its compression to $K_{e_\nu/e_{\mu_m}}$ shows that
 \[
 Q_A(e_{\mu_m} (\chi-\eta)f) = Q_A((\chi-\eta)f),\quad m\ge n
 \]
and therefore
\[
\begin{split}
0&= \frac{1}{\mu_m(\bbT)} \big( \<A (e_{\mu_m}(\chi-\eta)f) , e_{\mu_m}(\chi-\eta)f\>
- \<A ((\chi-\eta)f) , (\chi-\eta)f\>\big)\\
&=\left\< \frac{e_{\mu_m}-1}{\mu_m(\bbT)} (\chi-\eta)f   , A^*( e_{\mu_m}(\chi-\eta)f )\right\>\\
&\hskip3cm+\left\< A((\chi-\eta)f),  \frac{e_{\mu_m}-1}{\mu_m(\bbT)} (\chi -\eta)f    \right\>
\end{split}
\]
By Lemma~\ref{le:convergence}(ii) $(1/\mu_m(\mathbb{T}))({e_{\mu_m}-1}) (\chi-\eta)f $ tends weakly in $H^2$ to $(\chi+\eta)f$. On the other hand, $e_{\mu_m}(\chi-\eta)f$ tends pointwise on $\mathbb D$ to $(\chi-\eta)f$  and $\|e_{\mu_m}(\chi-\eta)f\|= \|(\chi-\eta)f\|$. By Lemma~\ref{le:weak convergence general}  $e_{\mu_m}(\chi-\eta)f$ tends to $(\chi-\eta)f$ in norm. We obtain then, by letting $m\to\infty$  the last equality,
\[
\< A((\chi+\eta)f), (\chi-\eta)f\>+\<A((\chi-\eta)f), (\chi+\eta)f\>=0.
\]
A simple calculation yields then
$
Q_A(Sf)=Q_A(f),
$ thereby concluding the proof.
\end{proof}

We observe that the argument above only requires that $A$ and its compressions to $K_{e_\nu/e_{\mu_n}}$, $n\in\mathbb N$, be complex symmetric.

\end{document}